\documentclass{amsart}

\usepackage{color}

\newtheorem{theorem}{Theorem}[section]
\newtheorem{lemma}[theorem]{Lemma}

\newtheorem{corollary}[theorem]{Corollary}

\theoremstyle{definition}

\theoremstyle{remark}
\newtheorem{remark}[theorem]{Remark}

\numberwithin{equation}{section}

\usepackage[T1]{fontenc}
\usepackage[utf8]{inputenc}
\usepackage{lmodern}

\usepackage{setspace}

\DeclareMathOperator{\tr}{tr}

\DeclareMathOperator{\Ric}{Ric}

\DeclareMathOperator{\can}{can}
\DeclareMathOperator{\II}{II}
\DeclareMathOperator{\ind}{ind}
\DeclareMathOperator{\Aut}{Aut}
\DeclareMathOperator{\diver}{div}
\DeclareMathOperator{\diam}{diam}
\DeclareMathOperator{\area}{area}

\newcommand{\R}{\mathbb{R}}
\newcommand{\C}{\mathbb{C}}
\newcommand{\D}{\mathbb{D}}

\newcommand{\s}{\mathbb{S}}
\newcommand{\I}{\mathcal{I}}

\begin{document}


\title[Rigidity of Free Boundary Surfaces in Compact 3-Manifolds]{Rigidity of Free Boundary Surfaces in Compact $3$-Manifolds with Strictly Convex Boundary}

\author{Abraão Mendes}
\address{Instituto de Matemática, Universidade Federal de Alagoas, Maceió, Alagoas, Brazil}
\email{abraao.rego@im.ufal.br}
\thanks{This work was carried out while the author was a Visiting Graduate Student at Princeton University during the 2015-2016 academic year. He was partially supported by NSF grant DMS-1104592 and by the CAPES Foundation, Ministry of Education of Brazil. He would like to express his gratitude to his Ph.D. advisors Fernando Codá Marques, at Princeton University, and Marcos Petrúcio Cavalcante, at UFAL. He also would like to thank the referee for the comments and suggestions.}

\date{\today}

\begin{abstract}
In this paper we obtain an analogue of Toponogov theorem in dimension 3 for compact manifolds $M^3$ with nonnegative Ricci curvature and strictly convex boundary $\partial M$. Here we obtain a sharp upper bound for the length $L(\partial\Sigma)$ of the boundary $\partial\Sigma$ of a free boundary minimal surface $\Sigma^2$ in $M^3$ in terms of the genus of $\Sigma$ and the number of connected components of $\partial\Sigma$, assuming $\Sigma$ has index one. After, under a natural hypothesis on the geometry of $M$ along $\partial M$, we prove that if $L(\partial\Sigma)$ saturates the respective upper bound, then $M^3$ is isometric to the Euclidean 3-ball and $\Sigma^2$ is isometric to the Euclidean disk. In particular, we get a sharp upper bound for the area of $\Sigma$, when $M^3$ is a strictly convex body in $\R^3$, which is saturated only on the Euclidean 3-balls (by the Euclidean disks). We also consider similar results for free boundary stable CMC surfaces.
\end{abstract}

\maketitle

\section{Introduction}

A classical result due to Toponogov \cite{Toponogov} (see \cite{HangWang} for an alternative proof) says that if $M^2$ is a closed Riemannian surface with Gaussian curvature $K\ge1$, then the length of any closed simple geodesic $\gamma\subset M^2$ satisfies $L(\gamma)\le2\pi$. Furthermore, if $L(\gamma)=2\pi$, then $M^2$ is isometric to the standard unit 2-sphere $\s^2$.

In order to obtain a version of Toponogov theorem in dimension $3$, Bray, Brendle, Eichmair, and Neves \cite{BrayBrendleEichmairNeves} considered a real projective plane $\Sigma^2$ embedded into a compact Riemannian $3$-manifold $M^3$. They proved that if $\Sigma$ has least area among all real projective planes embedded into $M$, and $M$ has scalar curvature $R\ge6$, then the area of $\Sigma$ satisfies $A(\Sigma)\le2\pi$. Moreover, if $A(\Sigma)=2\pi$, then $M^3$ is isometric to $\R\mathbb{P}^3$ endowed with the canonical metric.

A few months later, Bray, Brendle, and Neves \cite{BrayBrendleNeves} considered the infimum of all homotopically non-trivial 2-spheres in a compact Riemannian 3-manifold $(M^3,g)$ with $\pi_2(M)\neq0$. In fact, if $\mathcal F$ denotes the set of all smooth maps $f:S^2\to M$ which represent a non-trivial element of $\pi_2(M)$ and $$\mathcal A(M,g):=\inf\{\area(S^2,f^*g):f\in\mathcal F\},$$ they proved that $$\mathcal A(M,g)\inf_MR\le8\pi,$$ where $R$ is the scalar curvature of $(M,g)$. Furthermore, if equality holds, then the universal cover of $(M^3,g)$ is isometric to $\R\times\s^2$ up to scaling. See \cite{CaiGalloway} and \cite{Nunes} for similar results.

In a more recent work, Marques and Neves \cite{MarquesNeves} considered the case of unstable minimal $2$-spheres. Among other things, they proved that if $\langle\,,\,\rangle$ is a Riemannian metric on $S^3$ with scalar curvature $R\ge6$, but $\langle\,,\,\rangle$ does not have constant sectional curvature one, then there exists a minimal $2$-sphere $\Sigma^2$ embedded into $M^3=(S^3,\langle\,,\,\rangle)$ satisfying $A(\Sigma)<4\pi$. Also, $\Sigma$ has index zero or one. This can be seen as an analogue of Toponogov theorem in dimension $3$, since, in general, there is no area bound for minimal $2$-spheres in $M^3$, as pointed out in \cite{MarquesNeves}.

Our goal in this work is to obtain a version of Toponogov theorem in dimension $3$ for compact manifolds with nonempty boundary. Before stating our results, let us remember an important result in the setting.

Let $M^3$ be a compact Riemannian $3$-manifold with nonempty boundary $\partial M$. Denote by $\mathcal F_M$ the set of all immersed disks in $M$ whose boundaries are homotopically nontrivial curves in $\partial M$. If $\mathcal F_M\neq\emptyset$, define
\begin{eqnarray*}
\mathcal A(M)=\inf_{\Sigma\in\mathcal F_M}A(\Sigma)\,\,\,\mbox{and}\,\,\,\mathcal L(M)=\inf_{\Sigma\in\mathcal F_M}L(\partial\Sigma).
\end{eqnarray*}

\begin{theorem}[Ambrozio, \cite{Ambrozio}]\label{theorem.Ambrozio}
Let $M^3$ be a compact Riemannian $3$-manifold with nonempty boundary $\partial M$. Assume that $\partial M$ is mean convex and $\mathcal F_M\neq\emptyset$. Then, 
\begin{eqnarray*}
\frac{1}{2}\mathcal A(M)\inf_MR+\mathcal L(M)\inf_{\partial M}H^{\partial M}\le2\pi,
\end{eqnarray*} 
where $R$ is the scalar curvature of $M$ and $H^{\partial M}$ is the mean curvature of $\partial M$. Furthermore, if equality holds, the universal cover of $M$ is isometric to $\R\times\Sigma_0$, where $\Sigma_0$ is the disk with constant Gaussian curvature $\inf_MR/2$ whose boundary $\partial\Sigma_0$ has constant geodesic curvature $\inf_{\partial M}H$.  
\end{theorem}

\begin{remark}
An immediate consequence of Ambrozio theorem is that if $\inf_MR=0$, $\inf_{\partial M}H^{\partial M}=1$, and $\mathcal L(M)=2\pi$, then the universal cover of $M$ is isometric to $\R\times\bar\D$, where $\bar\D$ is the unit disk in $\R^2$ endowed with the canonical metric. 
\end{remark}

Observe that Ambrozio's result is an analogue of Bray-Brendle-Neves theorem for $3$-manifolds with nonempty boundary. Motivated by Marques and Neves' work \cite{MarquesNeves}, we consider the case of unstable minimal surfaces in $3$-manifolds with nonempty boundary. 

Now, let us state our first result. Definitions will be given in Section \ref{Section.2}.

\begin{theorem}[Theorem \ref{theorem.1}]\label{theorem.1.3}
Let $M^3$ be a compact Riemannian 3-manifold with nonempty boundary $\partial M$. Suppose that $\Ric\ge0$ and $\II\ge1$,  where $\Ric$ is the Ricci tensor of $M$ and $\II$ is the second fundamental form of $\partial M$. If $\Sigma^2$ is a properly embedded free boundary minimal surface of index one in $M^3$, then the length of $\partial\Sigma$ satisfies
\begin{eqnarray}\label{eq.theorem.1.3}
L(\partial\Sigma)\le2\pi(g+r),
\end{eqnarray}
where $g$ is the genus of $\Sigma$ and $r$ is the number of connected components of $\partial\Sigma$. Moreover, if equality holds, we have:
\begin{enumerate}
\item[\rm (i)] $\Sigma$ (w.r.t. the induced metric from $M$) is isometric to the Euclidean unit disk $\bar\D$;
\item[\rm (ii)] $\partial\Sigma$ is a geodesic of $\partial M$;
\item[\rm (iii)] $\Sigma$ is totally geodesic in $M$; and
\item[\rm (iv)] all sectional curvatures of $M$ vanish on $\Sigma$.
\end{enumerate}
\end{theorem}

In \cite{FraserSchoen}, Fraser and Schoen proved that if $\Sigma^2$ is a compact orientable surface with nonempty boundary, then $\sigma_1(\Sigma)L(\partial\Sigma)\le2\pi(g+r)$, where $\sigma_1(\Sigma)$ is the first nonzero Steklov eigenvalue of $\Sigma$. On the other hand, Fraser and Li \cite{FraserLi} proved that if $\Ric\ge0$, $\II\ge1$, and $\Sigma^2$ is a properly embedded minimal surface in $M^3$ with free boundary in $\partial M$, then $\sigma_1(\Sigma)\ge1/2$. As a corollary, they obtained that $L(\partial\Sigma)\le4\pi(g+r)$. However, this bound is not sharp. Thus, Theorem \ref{theorem.1.3} is an improvement of Fraser and Li's result to a sharp upper bound when we assume that $\Sigma$ has index one.

If we make an extra assumption on the geometry of $M$ along $\partial M$, we can characterize the global geometry of $M$ when equality in (\ref{eq.theorem.1.3}) holds.

\begin{theorem}[Corollary \ref{corollary.1}]\label{theorem.1.4}
Let $M^3$ be a compact Riemannian 3-manifold with nonempty boundary $\partial M$. Suppose that $\Ric\ge0$, $\II\ge1$, and $K_M(T_p\partial M)\ge0$ for all $p\in\partial M$, where $K_M$ is the sectional curvature of $M$. If $\Sigma^2$ is a properly embedded free boundary minimal surface of index one in $M^3$, then the length of $\partial\Sigma$ satisfies
\begin{eqnarray*}
L(\partial\Sigma)\le2\pi(g+r).
\end{eqnarray*}
Furthermore, if equality holds, $M^3$ is isometric to the Euclidean unit 3-ball $\bar B^3$ and $\Sigma^2$ is isometric to the Euclidean unit disk $\bar\D$.
\end{theorem}

Using Theorem \ref{theorem.1.4} together with the isoperimetric inequality for minimal disks in $\R^3$ (see \cite{BarbosadoCarmo} for the general case), we have a sharp upper bound for the area of a properly embedded free boundary minimal disk of index one in a strictly convex domain in $\R^3$.

\begin{theorem}[Corollary \ref{corollary.2}]\label{theorem.1.5}
Let $\Omega$ be a smooth bounded domain in $\R^3$ whose boundary $\partial\Omega$ is strictly convex, say $\II\ge 1$, where $\II$ is the second fundamental form of $\partial\Omega$ in $\R^3$. If $\Sigma^2$ is a properly embedded free boundary minimal disk of index one in $\Omega$, then the area of $\Sigma$ satisfies 
\begin{eqnarray*}
A(\Sigma)\le\pi.
\end{eqnarray*} 
Moreover, if equality holds, $\Omega$ is the Euclidean unit 3-ball and $\Sigma^2$ is the Euclidean unit disk. 
\end{theorem}

For the general case of a free boundary minimal surface (not necessarily a disk) of index one in a strictly convex domain $\Omega$ in $\R^3$, we introduce the constant 
\begin{eqnarray*}
\mathcal{R}(\Omega)=\inf_{y\in\Omega}\sup_{x\in\partial\Omega}|x-y|,
\end{eqnarray*}
and have the following result.

\begin{theorem}[Corollary \ref{corollary.3}]\label{theorem.1.6}
Let $\Omega$ be a smooth bounded domain in $\R^3$ whose boundary $\partial\Omega$ is strictly convex, say $\II\ge 1$. If $\Sigma^2$ is a properly embedded free boundary minimal surface of index one in $\Omega$, then the area of $\Sigma$ satisfies 
\begin{eqnarray*}
A(\Sigma)\le\pi(g+r)\mathcal{R}(\Omega).
\end{eqnarray*} 
Moreover, if equality holds, $\Omega$ is the Euclidean unit 3-ball and $\Sigma^2$ is the Euclidean unit disk. 
\end{theorem}

It would be interesting to know if it is true that when $M^3$ satisfies $\Ric\ge0$ and $\II\ge1$, but $M^3$ is not isometric to $\bar B^3$, there exists a properly embedded free boundary minimal surface $\Sigma^2$ of index one in $M^3$ satisfying $L(\partial\Sigma)<2\pi(g+r)$.

In Section \ref{Section.3} we obtain similar results to Theorems \ref{theorem.1.3} and \ref{theorem.1.4} for free boundary stable CMC surfaces. We point out that Theorems \ref{theorem.1.5} and \ref{theorem.1.6} are also true for free boundary stable CMC surfaces assuming they are minimal.

\section{Free boundary minimal surfaces of index one}\label{Section.2}

Let $M^3$ be a compact connected Riemannian 3-manifold with nonempty boundary $\partial M$. In this work, we assume that $M$ has nonnegative Ricci curvature and that $\partial M$ is strictly convex, which means $\II(V,V)=\langle D_VX,V\rangle>0$ for all $V\in T_p\partial M\setminus\{0\}$ and $p\in\partial M$, where $X$ is the outward pointing unit normal to $\partial M$ and $D$ is the Levi-Civita connection of $M$. Here, $\II$ is the second fundamental form of $\partial M$ in $M$. Under these hypotheses, by \cite[Theorem 2.11]{FraserLi}, $M^3$ is diffeomorphic to the Euclidean unit 3-ball $\bar B^3$. In particular, $M$ is orientable.

Let $\Sigma^2$ be a compact surface with nonempty boundary $\partial\Sigma$. Suppose $\Sigma^2$ is properly embedded into $M^3$, i.e., $\Sigma^2$ is embedded into $M^3$ and $\Sigma\cap\partial M=\partial\Sigma$. Since $M^3$ is diffeomorphic to the unit ball $\bar B^3$, which is simply connected, $\Sigma$ must be orientable. Fix a unit normal to $\Sigma$, say $N$, and denote by $A$ the second fundamental form of $\Sigma$, that is, $A(Y,Z)=\langle D_YN,Z\rangle$, $Y,Z\in T_x\Sigma$, $x\in\Sigma$. Also, denote by $\nu$ the outward pointing conormal along $\partial\Sigma$ in $\Sigma$. We say that $\Sigma$ is free boundary if $\Sigma$ meets $\partial M$ orthogonally. In other words, $\Sigma$ is free boundary if $\nu=X$ along $\partial\Sigma$. 

Let $t\longmapsto\Sigma_t$, $t\in(-\varepsilon,\varepsilon)$, be a variation of $\Sigma=\Sigma_0$. It is well known that the first variation of area is given by 
\begin{eqnarray}\label{eq.1}
\frac{d}{dt}\bigg|_{t=0}A(\Sigma_t)=\int_{\Sigma}\diver_\Sigma(\xi)d\sigma=\int_{\Sigma}H\phi d\sigma+\int_{\partial\Sigma}\langle\xi,\nu\rangle ds,
\end{eqnarray}
where $\xi=\frac{\partial}{\partial t}|_{t=0}$ is the variation vector field, $\phi=\langle\xi,N\rangle$, and $H=\tr A$ is the mean curvature of $\Sigma$ in $M$. It follows from (\ref{eq.1}) that $\Sigma$ is a critical point for the area functional for variations that preserve the property $\Sigma\cap\partial M=\partial\Sigma$ if and only if $\Sigma$ is minimal with free boundary. Also, if $\Sigma$ is minimal with free boundary, the second variation of area is given by 
\begin{eqnarray*}
\frac{d^2}{dt^2}\bigg|_{t=0}A(\Sigma_t)=\I(\phi,\phi),
\end{eqnarray*}
where $\I:C^\infty(\Sigma)\times C^\infty(\Sigma)\to\R$ is the index form of $\Sigma$ given by 
\begin{eqnarray*}
&\I(\psi,\phi)=-\displaystyle\int_{\Sigma}\psi\{\Delta\phi+(\Ric(N,N)+|A|^2)\phi\}d\sigma+\int_{\partial\Sigma}\psi\left\{\frac{\partial\phi}{\partial\nu}-\II(N,N)\phi\right\}ds.&
\end{eqnarray*}
Above, $\Ric$ is the Ricci tensor of $M$ and $\Delta$ is the Laplace operator of $\Sigma$ with respect to the induced metric from $M$.

We say that $\phi\in C^\infty(\Sigma)$ is an eigenfunction of $\I$ associated to the eigenvalue $\lambda\in\R$ if $\I(\psi,\phi)=\lambda\langle\psi,\phi\rangle_{L^2(\Sigma)}$ for all $\psi\in C^\infty(\Sigma)$. This is equivalent to saying that $\phi$ solves the Robin-type boundary value problem
$$
\left\{
\begin{array}{ll}
L\phi+\lambda\phi=0&\mbox{on}\,\,\Sigma,\\
\dfrac{\partial\phi}{\partial\nu}=\II(N,N)\phi&\mbox{along}\,\,\partial\Sigma,
\end{array}
\right.
$$
where $L=\Delta+(\Ric(N,N)+|A|^2)$ is the Jacobi operator of $\Sigma$. If $\Sigma$ is minimal with free boundary, the index of $\Sigma$ is defined as the number of negative eigenvalues of $\I$ counted with multiplicities. The index of $\Sigma$ is denoted by $\ind(\Sigma)$. It is well known that the first eigenvalue $\lambda_1$ of $\I$ is characterized by the Rayleigh formula 
\begin{eqnarray}\label{eq.4}
\lambda_1=\inf_{\phi\in C^\infty(\Sigma)\setminus\{0\}}\frac{\I(\phi,\phi)}{\int_{\Sigma}\phi^2d\sigma}.
\end{eqnarray}
Thus, it follows directly from (\ref{eq.4}) that, under the assumptions $\Ric\ge0$ and $\II>0$, all free boundary minimal surfaces have index at least one, since $\I(1,1)<0$.

Before proving our first result, we are going to state a very important lemma. This lemma is based on an argument presented in \cite{Hersch} (see also \cite{LiYau}).

\begin{lemma}\label{lemma.Hersch}
Let $\Sigma^2$ be a compact Riemannian surface with nonempty boundary $\partial\Sigma$. Suppose that $F:\Sigma\to\bar\D$ and $\phi_1:\Sigma\to\R$ are continuous functions such that $F(\Sigma\setminus\partial\Sigma)\subset\D$ and $\phi_1\ge0$. Then, there exists $h\in\Aut(\bar\D)$ such that $$\int_\Sigma(h\circ F)\phi_1d\sigma=0.$$
\end{lemma}

\begin{proof}
See Appendix A.
\end{proof}

Our first result is the following.

\begin{theorem}\label{theorem.1}
Let $M^3$ be a compact Riemannian 3-manifold with nonempty boundary $\partial M$. Suppose that $\Ric\ge0$ and $\II\ge1$,  where $\Ric$ is the Ricci tensor of $M$ and $\II$ is the second fundamental form of $\partial M$. If $\Sigma^2$ is a properly embedded free boundary minimal surface of index one in $M^3$, then the length of $\partial\Sigma$ satisfies
\begin{eqnarray}\label{eq.theorem.1}
L(\partial\Sigma)\le2\pi(g+r),
\end{eqnarray}
where $g$ is the genus of $\Sigma$ and $r$ is the number of connected components of $\partial\Sigma$. Moreover, if equality holds, we have:
\begin{enumerate}
\item[\rm (i)] $\Sigma$ (w.r.t. the induced metric from $M$) is isometric to the Euclidean unit disk $\bar\D$;
\item[\rm (ii)] $\partial\Sigma$ is a geodesic of $\partial M$;
\item[\rm (iii)] $\Sigma$ is totally geodesic in $M$; and
\item[\rm (iv)] all sectional curvatures of $M$ vanish on $\Sigma$.
\end{enumerate}
\end{theorem}

\begin{proof}
Let $\phi_1:\Sigma\to\R$ be the first eigenfunction of $\I$. We know that $\phi_1$ does not change sign. Then, without loss of generality, we can assume $\phi_1\ge0$. Since $\ind(\Sigma)=1$, for all $f\in C^\infty(\Sigma)$ with $\int_\Sigma f\phi_1d\sigma=0$, we have $\I(f,f)\ge0$, i.e.,
\begin{eqnarray}\label{eq.aux.1}
\int_\Sigma\{|\nabla f|^2-(\Ric(N,N)+|A|^2)f^2\}d\sigma-\int_{\partial\Sigma}\II(N,N)f^2ds\ge0.
\end{eqnarray}
On the other hand, by \cite[Theorem 7.2]{Gabard}, there exists a proper conformal branched cover $F:\Sigma\to\bar\D$ satisfying $\deg(F)\le g+r$. By Lemma \ref{lemma.Hersch}, we can assume $\int_\Sigma f_i\phi_1d\sigma=0$, where $F=(f_1,f_2)$. Then, using $f_i$ ($i=1,2$) as test function in (\ref{eq.aux.1}), we have
\begin{eqnarray*}
0&\le&\int_\Sigma\{|\nabla f_i|^2-(\Ric(N,N)+|A|^2)f_i^2\}d\sigma-\int_{\partial\Sigma}\II(N,N)f_i^2ds\\
&\le&\int_\Sigma|\nabla f_i|^2d\sigma-\int_{\partial\Sigma}f_i^2ds,
\end{eqnarray*}
where above we have used that $\Ric\ge0$ and $\II\ge1$. Hence, because $F(\partial\Sigma)\subset\s^1$ (since $F$ is proper) and $F$ is conformal, 
\begin{eqnarray*}
0&\le&\sum_{i=1}^2\left(\int_\Sigma|\nabla f_i|^2d\sigma-\int_{\partial\Sigma}f_i^2ds\right)=2\int_\Sigma dF^*g_{\can}-L(\partial\Sigma)\\
&=&2\pi\deg(F)-L(\partial\Sigma)\le2\pi(g+r)-L(\partial\Sigma),
\end{eqnarray*}
which implies (\ref{eq.theorem.1}).

If equality in (\ref{eq.theorem.1}) holds, all inequalities above must be equalities. Then, $A\equiv0$, $\Ric(N,N)=0$ on $\Sigma$, and $\II(N,N)=1$ along $\partial\Sigma$. Using the Gauss equation $R+H^2-|A|^2=2(\Ric(N,N)+K)$, where $K$ is the Gaussian curvature of $\Sigma$ and $R$ is the scalar curvature of $M$, we have $2K=R\ge0$. Observe that, since $\Sigma$ is free boundary ($\nu=X$ along $\partial\Sigma$), the geodesic curvature of $\partial\Sigma$ in $\Sigma$ is given by $\kappa=g(D_T\nu,T)=g(D_TX,T)=\II(T,T)\ge1$, where $T$ is the unit tangent to $\partial\Sigma$. Then, by Gauss-Bonnet theorem,
\begin{eqnarray*}
2\pi(2-2g-r)&=&2\pi\chi(\Sigma)=\int_\Sigma Kd\sigma+\int_{\partial\Sigma}\kappa ds\\
&\ge&L(\partial\Sigma)=2\pi(g+r), 
\end{eqnarray*}
i.e.,
\begin{eqnarray*}
2\ge3g+2r,
\end{eqnarray*} 
which implies $r=1$ and $g=0$. Then, all inequalities above must be equalities. So, $K\equiv0$ and $\kappa\equiv1$. Also, observe that the geodesic curvature $\bar\kappa$ of $\partial\Sigma$ in $\partial M$ (w.r.t. $N$) satisfies $\bar\kappa=g(D_TN,T)=A(T,T)=0$, thus $\partial\Sigma$ is a geodesic of $\partial M$. Now, let $x\in\Sigma$ and $\{e_1,e_2,e_3=N\}\subset T_xM$ be such that $\{e_1,e_2\}$ is an orthonormal basis of $T_x\Sigma$ and denote by $K_M$ the sectional curvature of $M$. Since $\Ric(e_1,e_1)+\Ric(e_2,e_2)+\Ric(e_3,e_3)=R=0$ on $\Sigma$ and $\Ric\ge0$ everywhere, we have $\Ric(e_i,e_i)=0$ on $\Sigma$ for $i=1,2,3$, which implies $K_M(e_i,e_j)=0$ for $i\neq j$.
\end{proof}

Below, we are going to present some corollaries of Theorem \ref{theorem.1}. But, before doing that, let us state an important result due to Xia \cite{Xia}.

\begin{theorem}[Xia]
Let $M^{n+1}$ be a compact Riemannian $(n+1)$-manifold with nonempty boundary $\partial M$. Suppose that $\Ric\ge0$ and $\II\ge c>0$ for some constant $c>0$, where $\Ric$ is the Ricci tensor of $M$ and $\II$ is the second fundamental form of $\partial M$ in $M$. Then, the first nonzero eigenvalue of the Laplace operator acting on functions on $\partial M$ (w.r.t. the induced metric from $M$) satisfies 
\begin{eqnarray*}
\lambda_1\ge nc^2.
\end{eqnarray*}
Furthermore, the equality holds if and only if $M^{n+1}$ is isometric to the Euclidean $(n+1)$-ball of radius $1/c$. 
\end{theorem}

Our first corollary is the following.

\begin{corollary}\label{corollary.1}
Let $M^3$ be a compact Riemannian 3-manifold with nonempty boundary $\partial M$. Suppose that $\Ric\ge0$, $\II\ge1$, and $K_M(T_p\partial M)\ge0$ for all $p\in\partial M$, where $K_M$ is the sectional curvature of $M$. If $\Sigma^2$ is a properly embedded free boundary minimal surface of index one in $M^3$, then the length of $\partial\Sigma$ satisfies
\begin{eqnarray*}
L(\partial\Sigma)\le2\pi(g+r).
\end{eqnarray*}
Furthermore, if equality holds, $M^3$ is isometric to the Euclidean unit 3-ball $\bar B^3$ and $\Sigma^2$ is isometric to the Euclidean unit disk $\bar\D$.
\end{corollary}

\begin{proof}
Denote by $K_{\partial M}$ the Gaussian curvature of $\partial M$ (w.r.t. the induced metric from $M$). Also, denote by $k_1$ and $k_2$ the principal curvatures of $\partial M$ in $M$. By Gauss equation,
\begin{eqnarray*}
K_{\partial M}=K_M(T_p\partial M)+k_1k_2\ge 1.
\end{eqnarray*}
Now, if $L(\partial\Sigma)=2\pi(g+r)$, by Theorem \ref{theorem.1}, $\Sigma^2$ is isometric to $\bar\D$ and $\partial\Sigma$ is a geodesic of $\partial M$. In particular, $\partial\Sigma$ is a simple (because $\Sigma$ is embedded into $M$) geodesic of $\partial M$ with $L(\partial\Sigma)=2\pi$. Then, by Toponogov theorem, $\partial M$ is isometric to the standard unit 2-sphere $\s^2$. Thus, by Xia theorem, $M^3$ is isometric to $\bar B^3$.
\end{proof}

Using the corollary above together with the isoperimetric inequality for minimal disks in $\R^3$ (see \cite{BarbosadoCarmo} for the general case), we have a sharp upper bound for the area of a properly embedded free boundary minimal disk of index one in a strictly convex domain in $\R^3$.

\begin{corollary}\label{corollary.2}
Let $\Omega$ be a smooth bounded domain in $\R^3$ whose boundary $\partial\Omega$ is strictly convex, say $\II\ge 1$, where $\II$ is the second fundamental form of $\partial\Omega$ in $\R^3$. If $\Sigma^2$ is a properly embedded free boundary minimal disk of index one in $\Omega$, then the area of $\Sigma$ satisfies 
\begin{eqnarray*}
A(\Sigma)\le\pi.
\end{eqnarray*} 
Moreover, if equality holds, $\Omega$ is the Euclidean unit 3-ball and $\Sigma^2$ is the Euclidean unit disk. 
\end{corollary}

\begin{proof}
The isoperimetric inequality for minimal disks in $\R^3$ says that 
\begin{eqnarray*}
4\pi A(\Sigma)\le L(\partial\Sigma)^2
\end{eqnarray*}
Then, by Theorem \ref{theorem.1}, $A(\Sigma)\le L(\partial\Sigma)^2/(4\pi)\le\pi$. Moreover, if $A(\Sigma)=\pi$, then $L(\partial\Sigma)=2\pi$, which by Corollary \ref{corollary.1} implies that $\Omega$ is the Euclidean unit 3-ball and $\Sigma^2$ is the Euclidean unit disk. 
\end{proof}

For the general area estimate we will introduce a constant depending on the domain. For this purpose, let $\Omega$ be a smooth bounded domain in $\R^3$ whose boundary $\partial\Omega$ is strictly convex, say $\II\ge 1$. Define $\mathcal{R}(\Omega)$ by 
\begin{eqnarray*}
\mathcal{R}(\Omega)=\inf_{y\in\Omega}\sup_{x\in\partial\Omega}|x-y|.
\end{eqnarray*}
It is not difficult to see that 
\begin{eqnarray*}
\frac{\diam(\Omega)}{2}\le\mathcal{R}(\Omega)\le\diam(\Omega).
\end{eqnarray*}
 Moreover, since $\II\ge 1$, we have $K_{\partial\Omega}\ge 1$, which by Bonnet-Myers theorem implies that $\diam(\partial\Omega)\le\pi$. Then, 
\begin{eqnarray*}
\mathcal{R}(\Omega)\le\diam(\Omega)<\diam(\partial\Omega)\le\pi.
\end{eqnarray*} 
$\mathcal{R}(\Omega)>0$ is the smallest real number $\delta>0$ such that $\Omega\subset B^3(x,\delta)$ for some $x\in\Omega$, where $B^3(x,\delta)$ is the Euclidean 3-ball of radius $\delta>0$ and center $x$.

Our result is the following.

\begin{corollary}\label{corollary.3}
Let $\Omega$ be a smooth bounded domain in $\R^3$ whose boundary $\partial\Omega$ is strictly convex, say $\II\ge 1$. If $\Sigma^2$ is a properly embedded free boundary minimal surface of index one in $\Omega$, then the area of $\Sigma$ satisfies 
\begin{eqnarray}\label{eq.General.Area.Estimate}
A(\Sigma)\le\pi(g+r)\mathcal{R}(\Omega).
\end{eqnarray} 
Moreover, if equality holds, $\Omega$ is the Euclidean unit 3-ball and $\Sigma^2$ is the Euclidean unit disk. 
\end{corollary}

\begin{proof}
Let $y_0\in\Omega$ be such that $\sup_{x\in\partial\Omega}|x-y_0|=\mathcal{R}(\Omega)$. Define $f:
\Sigma\to\R$ by $f(x)=\frac{1}{2}|x-y_0|^2$. Since $\Sigma$ is minimal, we have $\Delta f=2$. Then, 
\begin{eqnarray*}
2A(\Sigma)&=&\int_\Sigma\Delta f d\sigma=\int_{\partial\Sigma}\frac{\partial f}{\partial\nu}ds\\
&=&\int_{\partial\Sigma}\langle x-y_0,\nu\rangle ds\le\int_{\partial\Sigma}|x-y_0|ds\\
&\le &\mathcal{R}(\Omega)L(\partial\Sigma).
\end{eqnarray*}
Therefore, using $L(\partial\Sigma)\le2\pi(g+r)$ into the last inequality above, we get (\ref{eq.General.Area.Estimate}).

Now, if $A(\Sigma)=\pi(g+r)\mathcal{R}(\Omega)$, then $L(\partial\Sigma)=2\pi(g+r)$. The result follows from Corollary \ref{corollary.1}.
\end{proof}

\section{Free boundary stable CMC surfaces}\label{Section.3}

In this section we obtain a similar result to Theorem \ref{theorem.1} for free boundary stable constant mean curvature (CMC) surfaces and some of its consequences.

As before, let $M^3$ be a compact Riemannian $3$-manifold with nonempty boundary $\partial M$. Assume that $\Ric\ge0$ and $\II>0$. Also, let $\Sigma^2$ be a properly embedded compact surface in $M^3$ with nonempty boundary $\partial\Sigma$. We say that $\Sigma$ is stationary if it is a critical point for the area functional for variations that preserve the property $\Sigma\cap\partial M=\partial\Sigma$ and are volume-preserving (see \cite{RosVergasta}). Equivalently, $\Sigma$ is stationary if it has constant mean curvature and is free boundary. A free boundary CMC surface $\Sigma$ is called stable if its second variation of area is nonnegative for variations as before, which is equivalent to saying that 
\begin{eqnarray}\label{eq.7}
\I(\phi,\phi)\ge0,
\end{eqnarray}
for all $\phi\in C^\infty(\Sigma)$ satisfying $\int_\Sigma\phi d\sigma=0$.

Our result is the following.

\begin{theorem}\label{theorem.2}
Let $M^3$ be a compact Riemannian 3-manifold with nonempty boundary $\partial M$. Suppose that $\Ric\ge0$ and $\II\ge1$. If $\Sigma^2$ is a properly embedded free boundary stable CMC surface in $M^3$, then the length of $\partial\Sigma$ satisfies
\begin{eqnarray}\label{eq.theorem.2}
L(\partial\Sigma)\le2\pi(g+r),
\end{eqnarray}
where $g$ is the genus of $\Sigma$ and $r$ is the number of connected components of $\partial\Sigma$. Moreover, if equality holds, we have:
\begin{enumerate}
\item[\rm (i)] $\Sigma$ (w.r.t. the induced metric from $M$) is isometric to the Euclidean unit disk $\bar\D$;
\item[\rm (ii)] $\partial\Sigma$ is a geodesic of $\partial M$;
\item[\rm (iii)] $\Sigma$ is totally geodesic in $M$; and
\item[\rm (iv)] all sectional curvatures of $M$ vanish on $\Sigma$.
\end{enumerate}
\end{theorem}

\begin{proof}
Since $\Sigma$ is stable, by (\ref{eq.7}), we have 
\begin{eqnarray*}
\int_\Sigma\{|\nabla f|^2-(\Ric(N,N)+|A|^2)f^2\}d\sigma-\int_{\partial\Sigma}\II(N,N)f^2ds\ge0
\end{eqnarray*}
for all $f\in C^\infty(\Sigma)$ satisfying $\int_\Sigma fd\sigma=0$. Let $F=(f_1,f_2):\Sigma\to\bar\D$ be a proper conformal branched cover as in the proof of Theorem \ref{theorem.1}. Using Lemma \ref{lemma.Hersch}, we can assume $\int_\Sigma f_id\sigma=0$. Then, 
\begin{eqnarray*}
0&\le&\sum_{i=1}^2\left\{\int_\Sigma\{|\nabla f_i|^2-(\Ric(N,N)+|A|^2)f_i^2\}d\sigma-\int_{\partial\Sigma}\II(N,N)f_i^2ds\right\}\\
&\le&2\pi(g+r)-L(\partial\Sigma),
\end{eqnarray*}
which proves (\ref{eq.theorem.2}). 
 
If equality holds, working exactly as in the proof of Theorem \ref{theorem.1}, we have the result.
\end{proof} 

The first consequence of Theorem \ref{theorem.2} is the following.

\begin{corollary}
Let $M^3$ be a compact Riemannian 3-manifold with nonempty boundary $\partial M$. Suppose that $\Ric\ge0$, $\II\ge1$, and $K_M(T_p\partial M)\ge0$ for all $p\in\partial M$, where $K_M$ is the sectional curvature of $M$. If $\Sigma^2$ is a properly embedded free boundary stable CMC surface in $M^3$, then the length of $\partial\Sigma$ satisfies
\begin{eqnarray*}
L(\partial\Sigma)\le2\pi(g+r).
\end{eqnarray*}
Furthermore, if equality holds, $M^3$ is isometric to the Euclidean unit 3-ball $\bar B^3$ and $\Sigma^2$ is isometric to the Euclidean unit disk $\bar\D$.
\end{corollary}

Below we have a characterization of the Euclidean unit 3-ball by the length of the boundary of properly embedded free boundary stable CMC disks in it.

\begin{corollary}
The only smooth bounded domain $\Omega\subset\R^3$, with boundary $\partial\Omega$ satisfying $\II\ge1$, which admits
a properly embedded free boundary stable CMC disk $\Sigma^2\subset\Omega$ with $L(\partial\Sigma)=2\pi$ is the unit ball.
\end{corollary}

Corollaries \ref{corollary.2} and \ref{corollary.3} are also true if we change the hypothesis ``minimal of index one'' by ``stable CMC and minimal''.

\begin{remark}
In \cite{RosVergasta}, Ros and Vergasta observed that the only free boundary minimal surfaces of index one in the ball $\bar B^3$ is the totally geodesic disks passing through the center of the ball. Also, they proved that the only free boundary stable CMC surfaces in $\bar B^3$ are the totally geodesic disks, the spherical caps or surfaces of genus 1 with embedded boundary having at most two connected components. Recently, Nunes \cite{Nunes2016} ruled out the existence of the latter kind of surfaces, i.e., he proved, among other things, that a stationary stable surface in $\bar B^3$ must have genus zero.
\end{remark}

\begin{remark}
By \cite{ChenFraserPang}, under the same assumptions of Theorem \ref{theorem.1}, $g=0,1$ and $r=1,2,3$ or $g=2,3$ and $r=1$ (in fact, it is enough $\partial M$ to be weakly convex). In the case of smooth bounded domains in $\R^3$ with strictly convex boundary, it follows from the index estimates obtained in \cite{AmbrozioCarlottoSharp} and \cite{Sargent} that $\ind(\Sigma)=1$ implies $g=0$ (and $r=1,2,3,4$) or $g=1$ and $r=1,2$. In the same case, applying the techniques of \cite{Nunes2016} for minimal surfaces of index one instead of free boundary stable CMC surfaces, we can see that $g=0,1$ if $\ind(\Sigma)=1$.
\end{remark}

\appendix

\section{Proof of Lemma {\ref{lemma.Hersch}}}

If $\phi_1=0$, the result is trivial. Then, without loss of generality, we can assume $\int_\Sigma\phi_1d\sigma=1$. Let $m_a\in\Aut(\bar\D)$ be given by 
\begin{eqnarray*}
m_a(z)=\frac{z-a}{1-\bar az},\,\,z\in\bar\D\subset\C,
\end{eqnarray*}
for each $a\in\D$. Define $f:\D\to\C$ by 
\begin{eqnarray*}
f(a)&=&\int_{\Sigma\setminus\partial\Sigma}(m_a\circ F)\phi_1d\sigma.
\end{eqnarray*}
It follows from Lebesgue dominated convergence theorem that $f$ is continuous. Now, we want to extend $f$ to $\bar\D$ continuously. For this purpose, first observe that, if $a\in\s^1$, 
\begin{eqnarray*}
\frac{z-a}{1-\bar az}=\frac{z-a}{a^{-1}(a-z)}=-a,
\end{eqnarray*} for all $z\in\D$. Then,
\begin{eqnarray*}
\int_{\Sigma\setminus\partial\Sigma}\frac{F-a}{1-\bar aF}\phi_1d\sigma=-a\int_{\Sigma\setminus\partial\Sigma}\phi_1d\sigma=-a,
\end{eqnarray*} 
where above we have used that $F(\Sigma\setminus\partial\Sigma)\subset\D$. Second, if $a_n\longrightarrow a\in\s^1$ with $a_n\in\D\setminus\{0\}$, we have 
\begin{eqnarray*}
m_{a_n}(z)=\frac{z-a_n}{1-\bar a_nz}=\frac{z-a_n}{a_n^{-1}(a_n-|a_n|^2z)}\longrightarrow\frac{z-a}{a^{-1}(a-z)}=-a,
\end{eqnarray*} 
for all $z\in\D$. Then, defining $f(a)=-a$ for $a\in\s^1$, by Lebesgue dominated convergence theorem, $f:\bar\D\to\C$ is continuous.

Now, observe that $|f(a)|\le 1$ for all $a\in\bar\D$. Then, $f:\bar\D\to\bar\D$ is a continuous function satisfying $f(a)=-a$ for $a\in\s^1$. Thus, by topological reasons, $f$ is onto. Therefore, there exists $a_0\in\D$ such that $f(a_0)=0$. Take $h=m_{a_0}$.

\bibliographystyle{amsplain}
\bibliography{bibliography.bib}

\providecommand{\bysame}{\leavevmode\hbox to3em{\hrulefill}\thinspace}
\providecommand{\MR}{\relax\ifhmode\unskip\space\fi MR }
\providecommand{\MRhref}[2]{%
  \href{http://www.ams.org/mathscinet-getitem?mr=#1}{#2}
}
\providecommand{\href}[2]{#2}
\begin{thebibliography}{10}

\bibitem{AmbrozioCarlottoSharp}
L.~Ambrozio, A.~Carlotto, and B.~Sharp, \emph{Index estimates for free boundary
  minimal hypersurfaces},
  \href{http://arxiv.org/abs/1605.09704}{arXiv:1605.09704v1} (2016).

\bibitem{Ambrozio}
L.~C. Ambrozio, \emph{Rigidity of area-minimizing free boundary surfaces in
  mean convex three-manifolds}, J. Geom. Anal. \textbf{25} (2015), no.~2,
  1001--1017. \MR{3319958}

\bibitem{BarbosadoCarmo}
J.~L. Barbosa and M.~do~Carmo, \emph{A proof of a general isoperimetric
  inequality for surfaces}, Math. Z. \textbf{162} (1978), no.~3, 245--261.
  \MR{508841}

\bibitem{BrayBrendleEichmairNeves}
H.~Bray, S.~Brendle, M.~Eichmair, and A.~Neves, \emph{Area-minimizing
  projective planes in 3-manifolds}, Comm. Pure Appl. Math. \textbf{63} (2010),
  no.~9, 1237--1247. \MR{2675487}

\bibitem{BrayBrendleNeves}
H.~Bray, S.~Brendle, and A.~Neves, \emph{Rigidity of area-minimizing
  two-spheres in three-manifolds}, Comm. Anal. Geom. \textbf{18} (2010), no.~4,
  821--830. \MR{2765731}

\bibitem{CaiGalloway}
M.~Cai and G.~J. Galloway, \emph{Rigidity of area minimizing tori in
  3-manifolds of nonnegative scalar curvature}, Comm. Anal. Geom. \textbf{8}
  (2000), no.~3, 565--573. \MR{1775139 (2001j:53051)}

\bibitem{ChenFraserPang}
J.~Chen, A.~Fraser, and C.~Pang, \emph{Minimal immersions of compact bordered
  {R}iemann surfaces with free boundary}, Trans. Amer. Math. Soc. \textbf{367}
  (2015), no.~4, 2487--2507. \MR{3301871}

\bibitem{FraserLi}
A.~Fraser and M.~M.-c. Li, \emph{Compactness of the space of embedded minimal
  surfaces with free boundary in three-manifolds with nonnegative {R}icci
  curvature and convex boundary}, J. Differential Geom. \textbf{96} (2014),
  no.~2, 183--200. \MR{3178438}

\bibitem{FraserSchoen}
A.~Fraser and R.~Schoen, \emph{The first {S}teklov eigenvalue, conformal
  geometry, and minimal surfaces}, Adv. Math. \textbf{226} (2011), no.~5,
  4011--4030. \MR{2770439 (2012f:58054)}

\bibitem{Gabard}
A.~Gabard, \emph{Sur la repr\'esentation conforme des surfaces de {R}iemann \`a
  bord et une caract\'erisation des courbes s\'eparantes}, Comment. Math. Helv.
  \textbf{81} (2006), no.~4, 945--964. \MR{2271230 (2007k:14122)}

\bibitem{HangWang}
F.~Hang and X.~Wang, \emph{Rigidity theorems for compact manifolds with
  boundary and positive {R}icci curvature}, J. Geom. Anal. \textbf{19} (2009),
  no.~3, 628--642. \MR{2496569 (2010k:53065)}

\bibitem{Hersch}
J.~Hersch, \emph{Quatre propri\'et\'es isop\'erim\'etriques de membranes
  sph\'eriques homog\`enes}, C. R. Acad. Sci. Paris S\'er. A-B \textbf{270}
  (1970), A1645--A1648. \MR{0292357 (45 \#1444)}

\bibitem{LiYau}
P.~Li and S.~T. Yau, \emph{A new conformal invariant and its applications to
  the {W}illmore conjecture and the first eigenvalue of compact surfaces},
  Invent. Math. \textbf{69} (1982), no.~2, 269--291. \MR{674407 (84f:53049)}

\bibitem{MarquesNeves}
F.~C. Marques and A.~Neves, \emph{Rigidity of min-max minimal spheres in
  three-manifolds}, Duke Math. J. \textbf{161} (2012), no.~14, 2725--2752.
  \MR{2993139}

\bibitem{Nunes}
I.~Nunes, \emph{Rigidity of area-minimizing hyperbolic surfaces in
  three-manifolds}, J. Geom. Anal. \textbf{23} (2013), no.~3, 1290--1302.
  \MR{3078354}

\bibitem{Nunes2016}
\bysame, \emph{On stable constant mean curvature surfaces with free boundary},
  \href{http://arxiv.org/abs/1605.09625}{arXiv:1605.09625v1} (2016).

\bibitem{RosVergasta}
A.~Ros and E.~Vergasta, \emph{Stability for hypersurfaces of constant mean
  curvature with free boundary}, Geom. Dedicata \textbf{56} (1995), no.~1,
  19--33. \MR{1338315 (96h:53013)}

\bibitem{Sargent}
P.~Sargent, \emph{Index bounds for free boundary minimal surfaces of convex
  bodies}, \href{http://arxiv.org/abs/1605.09143}{arXiv:1605.09143v1} (2016).

\bibitem{Toponogov}
V.~A. Toponogov, \emph{Evaluation of the length of a closed geodesic on a
  convex surface}, Dokl. Akad. Nauk SSSR \textbf{124} (1959), 282--284.
  \MR{0102055 (21 \#850)}

\bibitem{Xia}
C.~Xia, \emph{Rigidity of compact manifolds with boundary and nonnegative
  {R}icci curvature}, Proc. Amer. Math. Soc. \textbf{125} (1997), no.~6,
  1801--1806. \MR{1415343 (97i:53043)}

\end{thebibliography}

\end{document}